\documentclass[12pt,oneside,reqno]{amsart}
\usepackage{amsmath,amsfonts,amssymb,amsthm,dsfont}
\usepackage{esint}
\usepackage{tikz}

\usepackage[a4paper,left=30mm,right=30mm,top=30mm,bottom=30mm,marginpar=25mm]{geometry} 
\usetikzlibrary{shadings,intersections,patterns}

\usepackage{pgfplots}

\pgfplotsset{compat=1.10}
\usepgfplotslibrary{fillbetween}

\usepackage{xfrac,xcolor}
\usetikzlibrary{positioning}
\usepackage{mathtools}
\usepackage[shortlabels]{enumitem}
\usepackage{hyperref} 
\usepackage[capitalise, nameinlink, noabbrev]{cleveref} 
\hypersetup{colorlinks=true, 
	linkcolor=red,
	citecolor=blue,
}

\DeclareMathOperator{\dive}{div}

\DeclareMathOperator{\per}{Per}
\DeclareMathOperator{\dist}{dist}

\def\ds{\displaystyle}
\def\eps{{\varepsilon}}
\def\N{\mathbb{N}}
\def\O{\Omega}
\def\R{\mathbb{R}}

\def\A{\mathcal{A}}
\def\mfA{\mathfrak{A}}

\def\HH{\mathcal{H}}

\def\M{\mathcal{M}}

\newcommand{\be}{\begin{equation}}
\newcommand{\ee}{\end{equation}}
\newcommand{\bib}[4]{\bibitem{#1}{\sc#2: }{\it#3. }{#4.}}
\newcommand{\cp}{\mathop{\rm cap}\nolimits}

\newcommand{\res}{\mathbin{\vrule height 1.6ex depth 0pt width 0.13ex\vrule height 0.13ex depth 0pt width 1.0ex}}
\newcommand{\sm}{\setminus}
\newcommand{\opt}{\mathrm{opt}}
\newcommand{\loc}{\mathrm{loc}}

\newcommand{\mathbbmm}[1]{\text{\usefont{U}{bbm}{m}{n}#1}}
\newcommand{\ind}{\mathbbmm{1}}

\numberwithin{equation}{section}
\theoremstyle{plain}
\newtheorem{theo}{Theorem}[section]
\newtheorem{lemm}[theo]{Lemma}
\newtheorem{coro}[theo]{Corollary}
\newtheorem{prop}[theo]{Proposition}
\newtheorem{defi}[theo]{Definition}

\theoremstyle{remark}
\newtheorem{rema}[theo]{Remark}
\newtheorem{exam}[theo]{Example}

\def\XXint#1#2#3{{\setbox0=\hbox{$#1{#2#3}{\int}$ }
\vcenter{\hbox{$#2#3$ }}\kern-.6\wd0}}

\title[Shape optimization problems in control form]{Shape optimization problems in control form}

\author[G. Buttazzo]{Giuseppe Buttazzo}

\author[F.P. Maiale]{Francesco Paolo Maiale}

\author[B. Velichkov]{Bozhidar Velichkov}

\date{}

\begin{document}

\maketitle

\hfill{\it Dedicated to the memory of Claudio Baiocchi}

\begin{abstract}
We consider a shape optimization problem written in the optimal control form: the governing operator is the $p$-Laplacian in the Euclidean space $\R^d$, the cost is of an integral type, and the control variable is the domain of the state equation. Conditions that guarantee the existence of an optimal domain will be discussed in various situations. It is proved that the optimal domains have a finite perimeter and, under some suitable assumptions, that they are open sets. A crucial difference is between the case $p>d$, where the existence occurs under very mild conditions, and the case $p\le d$, where additional assumptions have to be made on the data.
\end{abstract}



\section{Introduction}

Let $D\subset\R^d$ be an open set of finite measure; along all the paper this is considered as fixed. For every $\O\subset D$, which will be our control variable, the PDE which provides the state variable $u_\O$ of the system under observation is written as
\be\label{e:stateq}
-\Delta_pu_\O=f\quad\text{in }\O,\qquad u_\O=0\quad\text{on }\overline D\sm\O.
\ee
Here $p>1$ is given, $\Delta_p$ is the $p$-Laplacian
$$\Delta_pu=\dive\big(|\nabla u|^{p-2}\nabla u\big),$$
and $f\in W^{-1,p'}(D)$ is a prescribed right-hand side. Equivalently, $u_\O$ can be characterized as the unique solution of the variational minimum problem
\be\label{e:varpb}
\min\left\{\int_D\left(\frac1p|\nabla u|^p-f(x)u\right)\,dx\ :\ u\in W^{1,p}_0(D),\ u=0\hbox{ on }D\sm\O\right\}.
\ee
We consider shape optimization problems of the form
\be\label{e:minpb}
\min\big\{J(u_\O)\ :\ \O\text{ open},\ \O\subset D\big\},
\ee
for cost functionals $J$ of the form
\begin{equation}\label{e:definition-of-J}
J(u)=\int_D j\big(x,u\big)\,dx,
\end{equation}
where $j:D\times\R\to\R$ is a given integrand with $j(x,\cdot)$ lower semicontinuous. 
Our model case is 
$$j(x,s)=-g(x)s+\lambda\ind_{(0,+\infty)}(s),$$
where the function $g:D\to\R$ and the constant $\lambda>0$ are given. In this case,
\begin{equation}\label{e:modelJ}
J(u_\O)=-\int_Dg(x)u_\O\,dx+\lambda|\O|.
\end{equation}
Our main result is the following.

\begin{theo}\label{t:0}
Let $f\ge0$ be a function in $L^q(D)$, for some $q$ such that $q>d/p$ and $q\ge 1$. Let $\lambda>0$, $g\in L^r(D)$ for some $r>1$, be a non-negative measurable function and $J$ be the model functional given by \eqref{e:modelJ}.
\begin{enumerate}[\rm (i)]
\item If there is a constant $C>0$ such that 
$$f(x)\le Cg(x)\quad\text{for every}\quad x\in D,$$
then, there is an open set $\O_\opt\subset D$ solution of the problem \eqref{e:minpb}.
\item If $\O_\opt$ is a solution of \eqref{e:minpb}, then it has a finite perimeter. 
\end{enumerate}
\end{theo}	

When $f=g$ this problem reduces to a free boundary problem (see Section \ref{subsub:energy}) that is, the state function $u_\O$ on an optimal domain is also a minimizer of an integral functional which is defined on the while space $W^{1,p}_0(D)$. In particular, any function in $W^{1,p}_0(D)$ can be used to test the optimality of $u_\O$ and the properties of the optimal sets $\O$ can be studied through the known free boundary regularity techniques. On the other hand, when $f\neq g$ the problem \eqref{e:minpb} cannot be written as a variational problem in $W^{1,p}_0(D)$. In particular, this means that given an optimal shape $\O$, the corresponding state function $u_\O$ is a priori optimal only among the functions, which are state functions on other domains. This is the main difficulty in studying this functional, which was first studied in \cite{BV}, where the existence and some preliminary regularity properties of the minimizers were obtained in the case $p=2$.

\subsection{Structure of the paper and further results}
In this paper we prove several general results about the existence of optimal sets and the regularity of the state functions on solutions of the shape optimization problem \eqref{e:minpb}. The functionals we consider are of the form \eqref{e:definition-of-J} and our results apply to the model case \eqref{e:modelJ} (see Remark \ref{rem:pmagd}, Remark \ref{rem:exrel.quasiaperto}, Remark \ref{rem:exrel.aperto} and Remark \ref{rem:pfin}). In particular, Theorem \ref{t:0} is a consequence of the more general Theorem \ref{t:pmagd}, Theorem \ref{t:exrel.quasiaperto}, Theorem \ref{t:exrel.aperto} and Theorem \ref{t:pfin} below.

\subsubsection{Existence of optimal open sets in the case $p>d$}
First, in Theorem \ref{t:pmagd} we prove an existence result in the case $p>d$. Since in this case the continuity of the state functions is assured by the Sobolev embedding theorem, the existence of solutions in the class of open set is immediate and can be applied to a wide class of shape optimization problems. In particular, this theorem applies also to more general shape optimization problems of the form 
\be\label{e:minpb-measure-constraint}
\min\big\{J(u_\O)\ :\ \O\text{ open},\ \O\subset D,\ |\O|\le m\big\},
\ee
where $|\cdot|$ denotes the Lebesgue measure in $\R^d$, $m\in\big(0,|D|\big]$ is given and $J$ is a functional of the form \eqref{e:definition-of-J}.

\begin{theo}[Existence of optimal sets in the case $p>d$]\label{t:pmagd}
Let $p>d$ and suppose that the cost integrand $j$ satisfies the following condition:

For every $M>0$ there exists $a_M\in L^1(D)$ such that
\be\label{condj}
-a_M(x)\le j(x,s)\quad\text{for a.e. $x \in \R^d$ and all $|s|\le M$}.
\ee
Then, the minimization problem \eqref{e:minpb-measure-constraint} admits a solution $\O_\opt$ which is an open set.
\end{theo}

\begin{rema}\label{rem:pmagd}
The previous result applies for instance to functionals of the form \eqref{e:modelJ} with $g\in L^1(D)$ and $\lambda\in\R$. In this case, the function $a_M$ can be chosen as
$$a_M(x)=M|g(x)|+|\lambda|.$$
\end{rema}	

\begin{rema}
Notice that \eqref{e:minpb} corresponds to \eqref{e:minpb-measure-constraint} in which the measure constraint is set to be the trivial one: $m=|D|$. 	
\end{rema}

\subsubsection{Existence of optimal quasi-open sets in the case $p\le d$} In Theorem \ref{t:exrel.quasiaperto} we study shape optimization problems of the form \eqref{e:minpb} in the case $p\le d$. In this case, the existence of optimal sets may fail and additional assumptions on the integrand $j(x,s)$ are required to prove that an optimal domain $\O_\opt$ exists. Furthermore, the existence result cannot be obtained directly in the class of open sets; in fact, we consider the wider class of the so-called $p$-quasi-open sets (see Section \ref{spld}). Also, in this case the existence of optimal sets can be generalized to the following type of shape optimization problems with measure constraint:
\be\label{e:minpb-measure-constraint-qo}
\min\big\{J(u_\O)\ :\ \O\,\text{ $p$-quasi-open},\ \O\subset D,\ |\O|\le m\big\},
\ee
where $m\in\big(0,|D|\big]$ and $J$ is given by \eqref{e:definition-of-J}.

\begin{theo}[Existence of optimal quasi-open sets in the case $p\le d$]\label{t:exrel.quasiaperto}
Let $p\le d$ and $f\ge0$. Suppose in addition that the cost integrand $j$ satisfies the following conditions:
\begin{enumerate}[\rm(a)]
\item the function $j$ is of the form 
$$j(x,s)=j_0(x,s)+\lambda\ind_{(0,+\infty)}(s),$$
where $\lambda\ge0$ and, for Lebesgue almost-every $x\in\R^d$, the function $j_0(x,\cdot)$ is non-increasing;
\item there exists a function $a\in L^1(D)$ and a constant $c\in\R$ such that
$$a(x)-c|s|^r \le j(x,s)\,,\quad\text{where}\quad\begin{cases}\ds\ 0<r<\frac{dp}{d-p}\quad\text{when}\quad p<d\,,\smallskip\\
\ 0<r<+\infty\quad\text{when}\quad p=d\,.
\end{cases}$$
\end{enumerate}
Then, for any $m\in\big(0,|D|\big]$, there exists a solution $\O_\opt$ of \eqref{e:minpb-measure-constraint-qo}. Moreover, if $\lambda = 0$, then the measure constraint is saturated: $|\O_\opt|=m$.
\end{theo}

\begin{rema}\label{rem:exrel.quasiaperto}
It is immediate to check that the model functional \eqref{e:modelJ} satisfies the conditions of Theorem \ref{t:exrel.quasiaperto} when 
$$\lambda\ge0\ ,\qquad g\ge 0\qquad\text{and}\qquad g\in L^\ell(D)\quad\text{for some}\quad \ell>1.$$
The positivity of $g$ assures that (a) holds, while, for what concerns the condition (b), we can take 
$$c=1\ ,\quad r=\frac{\ell}{\ell-1}\quad\text{and}\quad a(x)=-|g(x)|^{\ell}.$$
In fact, by the Young's inequality
$$j(x,s)=-g(x)s+\lambda_{(0,+\infty)}(s)\ge -g(x)s\ge -|g(x)|^\ell-|s|^{\frac{\ell}{\ell-1}}=a(x)-s^r.$$
\end{rema}

\subsubsection{Existence of optimal open sets} In Theorem \ref{t:exrel.aperto} we show that the optimal quasi-open sets provided by Theorem \ref{t:exrel.quasiaperto} are in fact open. We obtain this result under some additional assumptions on the cost functional and on the datum $f$. In particular, we require that the function $j$ has some growth conditions, that the right-hand side $f$ in the state equation \eqref{e:stateq} has a suitable summability, and also that the problem is of the form \eqref{e:minpb}, that is, we remove the measure constraint in \eqref{e:minpb-measure-constraint-qo} by setting $m=|D|$. In particular, since the class of open sets is dense in the space of quasi-open sets, Theorem \ref{t:exrel.quasiaperto} and Theorem \ref{t:exrel.aperto} together provide an existence result for the shape optimization problem \eqref{e:minpb} in the class of open sets.

\begin{rema}\label{counterex}
Without extra summability assumptions on the right-hand side $f$, the optimal domain $\O_\opt$ in Theorem \ref{t:exrel.quasiaperto} is only a $p$-quasi open set, not open in general (see for instance \cite[Example 4.3]{BS20}).
\end{rema}

\begin{theo}[Optimal sets are open]\label{t:exrel.aperto}
Let $p\le d$, $m=|D|$ and $f\in L^q(D)$, where $q>d/p$. Suppose that the cost function $j$ is of the form 
$$j(x,s)=j_0(x,s)+\lambda\ind_{(0,+\infty)}(s),$$
where $\lambda> 0$ and $j_0$ satisfies the following assumptions: 
\begin{itemize}
\item $j_0(x,0)=0$ for almost-every $x\in D$;
\item there is  a constant $c>0$ such that
\be\label{hp.optk.g}
\frac{j_0(x,t)-j_0(x,s)}{t-s} \le -c f(x)
\ee
for almost-every $x \in D$ and all $s<t$.
\end{itemize}
Then, every solution of \eqref{e:minpb-measure-constraint-qo} is an open set. In particular, \eqref{e:minpb} has a solution in the class of all open subsets of $D$.
\end{theo}

\begin{rema}\label{rem:exrel.aperto}
It is immediate to check that the model functional \eqref{e:modelJ} satisfies the conditions of Theorem \ref{t:exrel.quasiaperto} when
$$g\ge cf\quad\text{on}\quad D\qquad\text{and}\qquad\lambda>0.$$
\end{rema}

\subsubsection{The optimal sets have finite perimeter}
In Theorem \ref{t:pfin}, we show that, under very mild assumptions, in both the situations $p>d$ and $p\le d$, the optimal domains are sets with a finite perimeter. The method we use was introduced by Bucur in \cite{bulbk} (see also \cite{bucve}) for the optimization of the $k$-th eigenvalue of the Laplacian; as we show, it can be applied to a much larger class of problems.

\begin{theo} \label{t:pfin}
Assume that $j$ is of the form
$$j(x,s)=j_0(x,s)+\lambda\ind_{(0,+\infty)}(s),$$
with $\lambda>0$ and 
\be\label{eq.jpos}
j_0(x,0)= 0\qquad \text{for a.e. }x\in D.
\ee
Let $\Omega$ be a solution to the problem \eqref{e:minpb-measure-constraint-qo} with $m=|D|$ (that is, without the measure constraint). Suppose that we are in either one of the following scenarios:
\begin{enumerate}[label=(\roman*)]
\item $f \in W^{-1,p'}(D)$, $f \ge 0$ and there exist $a \in L^1(D)$ and $\eps_0,c > 0$ such that
\be\label{eq.pf1}
\frac{|j_0(x,s+\eps)-j_0(x,s)|}{\eps}\le a(x)+c|s|^{p^\ast}
\ee
holds for all $s\in\R$, for a.e. $x \in D$ and for all $\eps\le\eps_0$;
\item $f\in L^q(D)$ for some $q>d/p$, $f \ge 0$ and there exist $a(\cdot,s)\in L^1(D)$, increasing and continuous in $s$, and $\eps_0>0$ such that
\be\label{eq.pf2}
\frac{|j_0(x,s+\eps)-j_0(x,s)|}{\eps}\le a(x,s)
\ee
holds for all $s\in\R$, a.e. $x \in D$ and and for all $\eps\le\eps_0$.
\end{enumerate}
Then the optimization problem \eqref{e:minpb-measure-constraint-qo} has a solution $\O_\opt$ which has finite perimeter.
\end{theo}

\begin{rema}
Notice that the assumption \eqref{eq.jpos} is not restrictive. In fact, by adding to the cost $J$ the constant quantity $\ds-\int_D j(x,0)\,dx$, we obtain that every optimal set for the cost function $j(x,s)$ is optimal also for the function $j(x,s)-j(x,0)$.
\end{rema}

\begin{rema}\label{rem:pfin}
The model functional \eqref{e:modelJ} satisfies the conditions of Theorem \ref{t:pfin} when $g\in L^1(D).$ In fact, we can simply take $a=|g|$.
\end{rema}

\subsection{Further remarks and related problems}

\subsubsection{The case $f=g$}\label{subsub:energy}
The existence and the regularity of optimal shapes in the case $f=g$ was already studied both for $p=2$ and $p\neq 2$. In fact, by testing the state equation \eqref{e:stateq} with $u_\O$, one gets that
$$\int_D|\nabla u_\O|^p\,dx=\int_{D}f(x)u_\O\,dx.$$ 
Thus, the functional $J$ can be written as 
$$J(u_\O):=\frac{p}{p-1}\left(\frac{1}{p}\int_D|\nabla u_\O|^p\,dx-\int_D f(x)u_\O\,dx+\frac{p-1}{p}\lambda|\O|\right).$$
Now, since the solution $u_\O$ of \eqref{e:stateq} is the minimizer of 
$$u\mapsto \frac{1}{p}\int_D|\nabla u|^p\,dx-\int_{D}f(x)u\,dx,$$
in $W^{1,p}_0(D)$, the shape optimization problem \eqref{e:minpb} becomes equivalent to the free boundary problem
\be\label{e:fbproblem}
\min\Big\{\frac{1}{p}\int_D|\nabla u|^p\,dx-\int_{D}f(x)u\,dx+\frac{p-1}{p}\lambda|\{u\neq 0\}|\ :\ u\in W^{1,p}_0(D)\Big\},
\ee
in the following sense:
\begin{itemize}
\item if $\O$ solves \eqref{e:minpb}, then $u_\O$ is a solution to \eqref{e:fbproblem};
\item if $u$ solves \eqref{e:fbproblem}, then the set $\{u\neq 0\}$ is optimal for \eqref{e:minpb}.
\end{itemize}
Problems of the form \eqref{e:fbproblem} were widely studied in the literature in the case $p=2$. The existence of solutions $u\in W^{1,p}$ is immediate and provides an existence of a solution to \eqref{e:minpb} in the class of the $p$-quasi-open sets. Also in this case, the existence of an optimal open set requires the study of the regularity of the solutions to \eqref{e:fbproblem}. For the general case $p\neq  2$, the regularity of $u$ and of the free boundary $\partial\{u>0\}$ were first studied by Danielli and Petrosyan in \cite{ddap1} in the case $f=0$. In the case $f\ge 0$, the problem was discussed in \cite{BS20}. For the case $p=2$ we refer to \cite{V15} and the references therein. \medskip

\subsubsection{Supremal functionals}
A class of shape optimization problems related to the one from Theorem \ref{t:0} are the ones, in which the cost functional is given by 
$$J(u_\Omega)=-\text{\rm ess}\,\sup_{\!\!\!\!\!\!\!\!\!\!\!\!\!x\in D}j(x,u_\Omega).$$
The existence of a solution can be obtained by the same argument as in the proofs of Theorem \ref{t:pmagd}, in the case $p>d$, and of Theorem \ref{t:exrel.quasiaperto}, in the case $p\le d$.
	
\subsubsection{Mixed boundary conditions}
The full Dirichlet boundary condition in \eqref{e:stateq} can be replaced by the mixed Dirichlet-Neumann condition
\be\label{mixed}
u=0\ \text{ on }\ D\sm\O,\qquad\frac{\partial u}{\partial\nu}=0\ \text{ on }\ \partial D\cap\partial\O.
\ee
The expression \eqref{mixed} is only formally written; we provide a weak form of it which allows to consider also very irregular sets. Notice that the Neumann condition in \eqref{mixed} is imposed on the {\it fixed} part of $\partial\O$, i.e. the part which lies on $\partial D$; on the {\it free} part of $\partial\O$, i.e. the one in $D$, we always assume the Dirichlet condition.
Precisely, the state function $u_\Omega$ is the minimizer of 
$$u\mapsto \frac1p\int_{\Omega}|\nabla u|^p\,dx-\int_{\Omega}uf\,dx,$$
among all functions  $u\in W^{1,p}(D)$ such that $u=0$ $p$-quasi-everywhere on $D\setminus \Omega$. In this case the existence results Theorem \ref{t:pmagd} and Theorem \ref{t:exrel.quasiaperto} still hold for the functional $J(u_\Omega)$ in the case $m<|D|$ with $D$ a bounded connected open set with Lipschitz boundary (see \cite[Section 2.2]{BV2}). On  the contrary, if 
$$m=|D|\qquad\text{and}\qquad f\ge 0,$$ 
then, $\Omega=D$ is admissible and $u_D$ may not be well-defined. In this case
$$\inf\big\{J(u_\O)\ :\ \O\text{ open},\ \O\subset D\big\}=-\infty,$$
for any non-trivial $g\ge 0$ and any $\lambda\ge0$.

When $m<|D|$, similar problems with mixed boundary conditions have been considered in \cite{BV2}. We point out that shape optimization problems with Neumann boundary condition on the free part require a complete different approach and very little is known on them.

\subsubsection{The limit problem as $p\to\infty$}
When $p=\infty$ the state function from \eqref{e:varpb} is determined by the variational problem
$$\min\left\{-\int_D f(x)u\,dx\ :\ u:D\to\R,\ |\nabla u|\le1,\ u=0\text{ on }\overline D\sm\O\right\},$$
whose solution is the function $u_\O$ given by the distance from $\overline D\sm\O$, namely
$$u_\O(x):=\dist(x,\overline D\sm\O).$$
In this way the problems above, with $p=\infty$, are related to some optimization problems in mass transport theory, see for instance \cite{BOS01} and \cite{BS03b}.

\section{Existence of a solution in the case $p>d$}\label{spgd}

We consider here the shape minimization problem \eqref{e:minpb} with the class $\A$ of admissible sets given by
$$\A=\big\{\O\subset D\ :\ \O\text{ open and }|\O|\le m\big\},$$
where $|\cdot|$ denotes the Lebesgue measure in $\R^d$ and $0<m\le|D|$ is given.

\begin{proof}[Proof of Theorem \ref{t:pmagd}]
Let $\O_n$ be a minimizing sequence in $\A$ and let $u_n$ be the solutions of \eqref{e:stateq} corresponding to $\O_n$. Multiplying \eqref{e:stateq} by $u_n$ and integrating by parts gives
$$\int_D|\nabla u_n|^p\,dx=\langle f,u_n\rangle\le\|f\|_{W^{-1,p'}(D)}\|u_n\|_{W^{1,p}(D)}.$$
On the other hand, by Poincar\'e inequality we obtain
$$\int_D|u_n|^p\,dx\le C\int_D|\nabla u_n|^p\,dx,$$
and the constant $C$ does not depend on $n$. Hence we have
$$\|u_n\|^p_{W^{1,p}(D)}\le(1+C)\int_D|\nabla u_n|^p\,dx\le(1+C)\|f\|_{W^{-1,p'}(D)}\|u_n\|_{W^{1,p}(D)},$$
so that $u_n$ is bounded in $W^{1,p}(D)$. By the Sobolev embedding theorem $u_n$ is compact in some $C^{0,\alpha}$-space, so we may assume, up to subsequences, that $u_n$ converges uniformly to some $C^{0,\alpha}$ function $u$. Then the set $\O_{\opt}=\{u\ne0\}$ is open, we have $\O_{\opt}\in\A$, and $u$ verifies the PDE \eqref{e:stateq} corresponding to $\O_{\opt}$. Moreover, by \eqref{condj} we have
$$J(u)\le\liminf_{n\to\infty}J(u_n),$$
which gives the optimality of the set $\O_{\opt}$ and concludes the proof.
\end{proof}

\begin{rema}
A similar result, with a similar proof, holds for minimization problems with the supremal cost $J(u)=-\|j(x,u)\|_{L^\infty(D)}$.
\end{rema}

We should not expect a symmetry result for solutions, in the sense that, even if all the data are radially symmetric, the solution $\O_{\opt}$ is not, as the next example shows:\medskip

\noindent\begin{minipage}{0,55\textwidth}\begin{exam}
Let $d=2$, $D$ the unit disc in $\R^2$, $p=\infty$ and consider the supremal optimization problem for $j(x,u)=-u$ and mixed Dirichlet-Neumann boundary conditions. Then $u_\O=\dist(x,D\sm\O)$ and the problem becomes
$$\min\left\{-\dist(x,D\sm\O)\ :\ |\O|\le m\right\}.$$
The optimal solution $\O_{\opt}$ (unique up to rotations) is represented in Figure 1 by the colored region and is given by the intersection
$$\O_{\opt}=B(0,1)\cap B(\bar{x},r_m)$$\end{exam}
\end{minipage}\quad\begin{minipage}{0,43\textwidth}
\centering
{\includegraphics[scale=0.7]{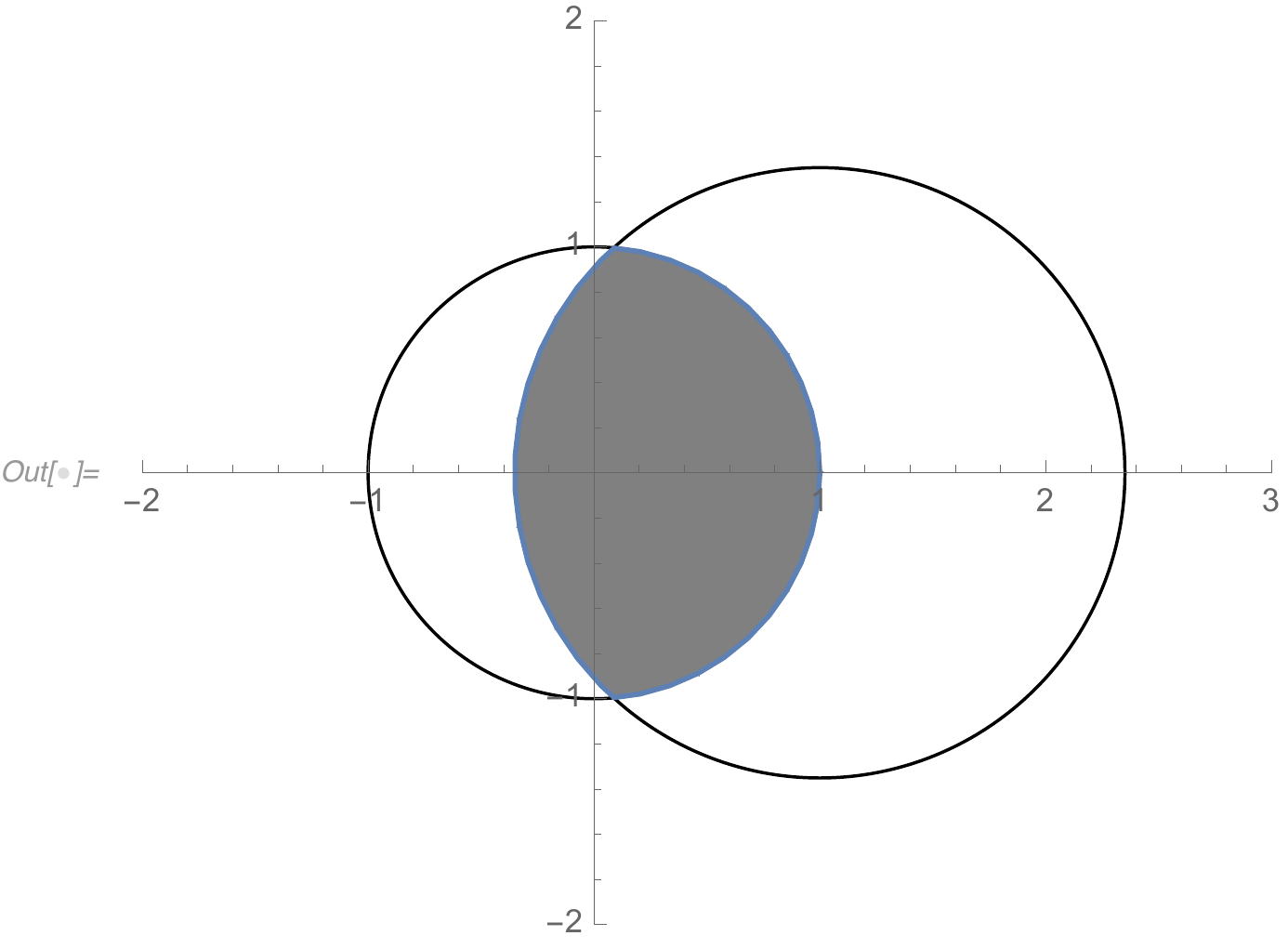}}\\
{\sc Figure 1. \rm The colored region represents an optimal $\O_{\opt}$ with $\bar{x}=(1,0)$ and $m\approx2$ ($r_m\approx1.351$).}
\end{minipage}
\medskip

\noindent
where $\bar{x}$ is any point in $\partial D$ and $r_m$ is the unique radius such that the area of the colored intersection is exactly equal to $m$.

\section{Existence in the case $p\le d$}\label{spld}

When $p\le d$ minimizing sequences $\O_n$ tend to split more and more, converging in a suitable $\gamma$-sense to a relaxed solution that is in general not a domain but only a {\it capacitary measure}. Several examples of nonexistence are known in the literature; we refer for instance to Section 4.2 of \cite{BB05}. For the sake of brevity we limit ourselves to the definitions of quasi open/closed set, of capacitary measure, and of $\gamma$-convergence; the reader interested to have a complete view on relaxed shape optimization problems with Dirichlet conditions on the free boundary and on capacitary measures may see, for instance, the book \cite{BB05} and the articles \cite{BDM91}, \cite{BDM93}.

\begin{defi}
A set $\O$ is said $p$-quasi open if $\O=\{u>0\}$ for a suitable function $u\in W^{1,p}(D)$. Similarly, a set $K$ is said $p$-quasi closed if $K=\{u=0\}$ for a suitable function $u\in W^{1,p}(D)$.
\end{defi}

\begin{defi}\label{capm}
A nonnegative regular Borel measure $\mu$ on $\R^d$ (possibly taking the value $+\infty$) is called a $p$-capacitary measure if for every Borel set $E\subset\R^d$ we have
\[\begin{cases}
\cp_p(E)=0\ \Longrightarrow\ \mu(E)=0;\\
\mu(E)=\inf\left\{\mu(\O)\: : \: \O\supset E,\ \text{$\O$ $p$-quasi open} \right\}.
\end{cases}\]
We denote by $\M_p$ the class of all $p$-capacitary measures on $\R^d$.
\end{defi}

\noindent The class $\M_p$ is very large; it includes:
\begin{enumerate}
\item[-] all measures of the form $a(x)\,dx$ with $a\in L^1_{\loc}(\R^d)$;
\item[-] all measures of the form $b(x)\,\HH^k\res S$ where $S$ is a smooth $k$-dimensional surface, $\HH^k$ is the $k$-dimensional Hausdorff measure, $k > d-p$, and $b$ is locally integrable on $S$;
\item[-] all measures of the form
$$\infty_K(E)=\begin{cases}
0&\text{if }\cp_p(E\cap K)=0\\
+\infty&\text{otherwise}
\end{cases}$$
where $K$ is a $p$-quasi closed set in $\R^d$.
\end{enumerate}
The important features of measures of the class $\M_p$ are:
\begin{enumerate}
\item[-] we can define the Sobolev space $W^{1,p}_\mu$ as the subspace of functions $u\in W^{1,p}(\R^d)$ with finite norm
\[
\|u\|_{W^{1,p}_\mu}=\left(\int_D|\nabla u|^p\,dx+\int |u|^p\,d\mu \right)^{1/p};
\]
\item[-] if $f$ is in the dual of $W^{1,p}_\mu$ (for instance if $f\in L^p(D)$), the PDE
\be\label{statrel}
-\Delta_p u+\mu|u|^{p-2}u=f
\ee
is well-defined in its weak sense as
$$\begin{cases}
u\in W^{1,p}_\mu\\
\ds\int_D|\nabla u|^{p-2}\nabla u\nabla v\,dx+\int|u|^{p-2}uv\,d\mu=\langle f,v\rangle\quad\text{for all $v\in W^{1,p}_\mu$}.
\end{cases}$$
The PDE above admits a unique solution $u_{\mu,f}$, which can be equivalently characterized as the unique minimum point of the functional
$$F(u)=\int_D\frac{1}{p}|\nabla u|^p\,dx+\int_D\frac{1}{p}|u|^p\,d\mu-\langle f,u\rangle.$$
Taking for instance $\O$ $p$-quasi open and $\mu=\infty_K$, with $K=D\sm\O$, the PDE above becomes \eqref{e:stateq} with Dirichlet conditions on $D\sm\O$, namely
$$\begin{cases}
-\Delta_p u=f\text{ on }\O\\
u=0\text{ on }K,\quad\frac{\partial u}{\partial\nu}=0\text{ on }\partial D\sm K.
\end{cases}$$
\end{enumerate}

If we denote by $u_{\mu,f}$ the unique solution of \eqref{statrel} the following monotonicity properties hold.

\begin{prop}\label{monotonia}
The map $\mu\mapsto u_{\mu,f}$ is decreasing with respect to $\mu$ whenever $f\ge0$, and increasing with respect to $f$, that is,
\[
\begin{cases}
\mu_1\ge\mu_2\Longrightarrow u_{\mu_1,f}\le u_{\mu_2,f}\qquad\text{for every }f\ge0;\\
f_1\le f_2\Longrightarrow u_{\mu,f_1}\le u_{\mu,f_2}\qquad\text{for every }\mu\in\M_p.
\end{cases}
\]
\end{prop}

\begin{proof}
Let $f \ge 0$ and set $u_i:=u_{\mu_i,f}$ for $i=1,2$. To prove the monotonicity with respect to $\mu$, that is $u_1\le u_2$ when $\mu_1\ge\mu_2$, we use the variational characterization of the solutions $u_i$ as minima of the corresponding functionals
\[
F_i(u)=\int_D\frac{1}{p} |\nabla u|^p\,dx + \int_D \frac{1}{p}|u|^p\,d\mu_i - \langle f,u\rangle,
\]
and we show, for example, that $F_1(u_1\wedge u_2)\le F_1(u_1)$ so by minimality and uniqueness $u_1 = u_1 \wedge u_2$, concluding the proof. Now notice that
\[
F_i(u\wedge v)+F_i(u\vee v)=F_i(u)+F_i(v)\quad \text{for $i=1,2$},
\]
so it is enough to show that
\[
F_1(u_2)\le F_1(u_1\vee u_2).
\]
By the equalities
\[
\begin{cases}
\ds F_1(u_2)=F_2(u_2)+\int_D \frac{1}{p} |u_2|^p\,d(\mu_1-\mu_2)\\
\ds F_1(u_1\vee u_2)=F_2(u_1\vee u_2)+\int_D\frac{1}{p} |u_1\vee u_2|^p\,d(\mu_1-\mu_2)
\end{cases}
\]
it is equivalent to show that
\be\label{eqz1}
F_2(u_2)\le F_2(u_1\vee u_2) + \int_D \frac{1}{p} \left(|u_1\vee u_2|^p-|u_2|^p\right) \, d(\mu_1-\mu_2). 
\ee
By the maximum principle, since $f\ge0$, we have $u_i\ge0$ for $i=1,2$, which gives \eqref{eqz1} using the minimality of $u_2$ for $F_2$ and the fact that $|u_1\vee u_2|\ge|u_2|$.

In a similar way we can prove the monotonicity with respect to $f$. Let $\mu \in \M_p$, set $u_i:=u_{\mu,f_i}$ for $i=1,2$ and consider the corresponding functionals given by 
\[
F_i(u)=\int_D\frac{1}{p}|\nabla u|^p\,dx+\int_D\frac{1}{p}|u|^p\,d\mu-\langle f_i,u\rangle.
\]
To prove the monotonicity property it is enough to show that
\[
F_1(u_1\wedge u_2)\le F_1(u_1),
\]
which, operating as above, amounts to be equivalent to
\be\label{eqz2}
F_1(u_2)\le F_1(u_1\vee u_2).
\ee
By the equalities
\[
\begin{cases}
F_1(u_2)=F_2(u_2)+\langle f_2-f_1,u_2\rangle\\
F_1(u_1\vee u_2)=F_2(u_1\vee u_2)+\langle f_2-f_1,u_1\vee u_2\rangle
\end{cases}
\]
the inequality \eqref{eqz2} can be rewritten as
\[
F_2(u_2)\le F_2(u_i\vee u_2)+\langle f_2-f_1,u_1\vee u_2-u_2\rangle,
\]
which follows by the minimality of $u_2$ for $F_2$ and the fact that $(u_1\vee u_2)-u_2\ge0$.
\end{proof}

We consider the subclass $\M_p(D)$ of $p$-capacitary measures supported on $\overline{D}$. This class can be endowed with a very natural convergence, the $\gamma_p$-convergence.

\begin{defi}\label{gammap}
A sequence $\mu_n\in\M_p(D)$ is said to $\gamma_p$-converge to $\mu\in\M_p(D)$ if
$$\|u_{\mu_n,1}-u_{\mu,1}\|_{L^p(D)}\to0.$$
\end{defi}

\begin{rema}
It is possible to show that
\[
\|u_{\mu_n,1}-u_{\mu,1}\|_{L^p(D)}\to0\ \implies\ \|u_{\mu_n,f}-u_{\mu,f}\|_{L^p(D)}\to0 \quad \text{for every $f\in W^{-1,p'}(D)$}.
\]
In addition, the distance defined by setting
\[
d_{\gamma_p}(\mu,\nu):=\|u_{\mu,1}-u_{\nu,1}\|_{L^p(D)}
\]
is equivalent to the $\gamma_p$-convergence and makes $\M_p(D)$ a compact metric space.
\end{rema}

An important result is the $\gamma_p$ density in $\M_p(D)$ of some subclasses:
\begin{enumerate}
\item[-] the class of measures $a(x)\,dx$ with $a$ smooth function;
\item[-] the class of measures $\infty_K$ with $K$ smooth closed set.
\end{enumerate}
In addition, if $\mu_n$ $\gamma_p$-converges to $\mu$, the Dirichlet regions $\{\mu_n=\infty\}$ fulfill the following semicontinuity property (see Proposition 5.3.6 of \cite{BB05}):
\be\label{scimeas}
\limsup_{n\to + \infty} |\{\mu_n=\infty\}|\le|\{\mu=\infty\}|.
\ee

\subsection{Existence of an optimal capacitary measure}

The relaxed form of the minimization problem \eqref{e:minpb} in the class $\M_p(D)$ reads
\be\label{minrel}
\min\left\{J(u_{\mu,f})\: :\:u_{\mu,f}\text{ solves \eqref{statrel}, }\mu\in\mfA\right\},
\ee
where $\mfA$ is the admissible class
\[
\mfA=\big\{\mu\in\M_p(D),\ |\{\mu<\infty\}|\le m\big\}.
\]
The framework above allows to obtain a rather general existence result of minimizers for the relaxed problem above. A similar argument was used in \cite{BMV} in the context of optimal potentials for Schr\"odinger operators. 

\begin{theo}\label{exrel}
Let $p\le d$ and let $f\in W^{-1,p'}(D)$. Assume that there exists $a \in L^1(D)$ such that the cost integrand $j$ satisfies the assumption
\[
a(x)-c|s|^q\le j(x,s)\qquad\text{with $q < \frac{dp}{d-p}$\quad\text{(any $q<+\infty$ if }p=d)}.
\]
Then the relaxed minimization problem \eqref{minrel} admits a solution $\mu_\opt\in\M_p$.
\end{theo}

\begin{proof}
If $\mu_n\in\A$ is a minimizing sequence, by the $\gamma_p$-compactness of $\M_p(D)$ we may assume, up to a subsequence, that $\mu_n$ $\gamma_p$-converges to some $\mu\in\M_p(D)$. By the lower semicontinuity property \eqref{scimeas} we have that
$$\limsup_{n\to+\infty}|\{\mu_n=\infty\}|\le|\{\mu=\infty\}|\implies|\{\mu<\infty\}|\le m,$$
which means that $\mu$ belongs to the admissible class $\mfA$. Then the corresponding solutions $u_{\mu_n,f}$ tend to $u_{\mu,f}$ weakly in $W^{1,p}(D)$, hence strongly in $L^q(D)$. By Fatou's lemma
\[
J(u_{\mu,f})\le\liminf_{n \to + \infty}J(u_{\mu_n,f}),
\]
and this concludes the proof.
\end{proof}

\subsection{Existence of $p$-quasi-open optimal sets}

In some cases the minimization problem \eqref{minrel} becomes trivial. For instance, if the cost integrand satisfies
\[
j(x,s)\ge j(x,0)\quad \text{for all $(x,s) \in \R^d\times \R$},
\]
then the measure $\mu=\infty_{\overline D}$, which corresponds to $\O_\opt$ being the empty set, gives $u_{\mu,f}=0$ and solves the problem. On the other hand, in some particular situations we can prove the existence of an optimal measure of the form $\mu=\infty_{D\setminus\O}$, as the following result shows.

\begin{proof}[Proof of Theorem \ref{t:exrel.quasiaperto}]
Let $\mu\in\M_p(D)$ be a relaxed optimal solution and let $u_{\mu,f}$ be the corresponding optimal state, solution of \eqref{statrel}. Since $f\ge0$, by the maximum principle we have $u_{\mu,f}\ge0$ and $\O =\{ u_{\mu,f} > 0 \}$ is a $p$-quasi open set. By \eqref{scimeas} we have 
\[
|\O|\le m\implies\tilde{\mu} = \infty_{D\sm\O} \in \mfA.
\]
We have $\tilde\mu\le\mu$ which gives $u_{\tilde\mu,f}\ge u_{\mu,f}$ by Proposition \ref{monotonia}. By the monotonicity assumption (a) on the integrand $j$ we obtain
\[
J(u_{\tilde\mu,f})\le J(u_{\mu,f})
\]
which shows that the measure $\infty_{D\sm\O}$ is optimal. Finally, if $|\O|<m$, taking $\hat{\O}\supset\O$ with $|\hat{\O}|=m$ would give $\infty_{D\sm\hat{\O}}\le\infty_{D\sm\O}$, and so $u_{\infty_{D\sm\hat{\O}},f}\ge u_{\infty_{D\sm\O},f}$, hence the optimality of $\hat{\O} =: \O_\opt$, using again the monotonicity property (a).
\end{proof}

\begin{rema}
The assumption that the integrand $j(x,\cdot)$ is nonincreasing cannot be removed; in Section 4.2 of \cite{BB05} the case $f=1$ and $j(x,s)=|s-c|^2$ is considered and it is shown that the optimal measures in $\M_p(D)$, when $c$ is small enough, are not of the form $\mu=\infty_{D\sm\O}$.
\end{rema}

\begin{lemm}\label{lemma.bdd1}
Let $\mu$ be a positive $p$-capacitary measure in $D$. Suppose that  $f\in L^q(D)$ for some $q>d/p$, then the solution $u_{\mu,f}$ is $L^\infty(D)$.
\end{lemm}

\begin{proof}
We can assume that $f\ge 0$. By Proposition \ref{monotonia}, we have that
\[
\mu\ge0\implies u_{\mu,f}\le u_{0,f},
\]
By definition, the function $u_{0,f}$ is the unique solution (in the weak sense) of the equation
\[
-\Delta_p u_{0,f}=f\text{ in }D.
\]
By modifying the proof of \cite[Theorem 8.17]{GT01} for any $p>1$, we find that
\[
\text{$f\in L^q(D)$ with }q>d/p\implies u_{0,f} \in L^\infty(D),
\]
and this is enough to infer that $u_{\mu,f} \in L^\infty(D)$.
\end{proof}

If we assume additional summability on the right-hand side $f$, then we can prove that the optimal set $\O_\opt$ is open. As mentioned in Remark \ref{counterex} these assumptions cannot be removed.

\subsection{Existence of open optimal sets}
Let $\Omega_\opt$ be a solution to \eqref{e:minpb-measure-constraint-qo} with $m=|D|$, as in Theorem \ref{t:exrel.aperto}.
To prove that $\O_\opt$ is open we will show that the corresponding solution $\bar{u}$ to the auxiliary problem
$$\min\Big\{\int_D \left( j(x,u) + \lambda\mathds{1}_{\{u>0\}}(x) \right) \,dx\ :\ \Delta_p u+f\ge0,\ u\in W_0^{1,p}(D)\Big\}$$
belongs to $C^{0,\alpha}(D)$ for some $\alpha < 1$; for this, we follow the proof of \cite[Section 3]{ddap1}.

\begin{proof}[Proof of Theorem \ref{t:exrel.aperto}]
Since we are dealing with the problem \eqref{e:minpb-measure-constraint-qo} the cost function is of the form
$$j(x,u)+\lambda\mathds{1}_{\{u>0\}},$$
where the constant $\lambda>0$ can be interpreted as a Lagrange multiplier. For the sake of simplicity we assume that $\lambda=1$.

Let $w$ be a function in $W^{1,p}_0(D)$ such that $w\ge \bar u$, where $\bar u$ is the solution of 
$$-\Delta_p\bar u=f\quad\text{in}\quad \O_\opt,\qquad\bar u\in W^{1,p}_0(\O_\opt)$$ 
on the optimal set $\O_\opt$. We define the quasi-open set 
$$\widetilde\O= \{w > 0\}.$$
Notice that, by construction, $\widetilde{\O}\supset\O_\opt$. Let $\widetilde u$ be the solution of 
$$-\Delta_p\widetilde u=f\quad\text{in}\quad \widetilde\O,\qquad \widetilde u\in W^{1,p}_0(\widetilde \O).$$ 
Thus, by the maximum principle $\widetilde u\ge \bar u$.
Since $\O_\opt$ is optimal, we have
$$ \int_{\O_\opt} j(x,\bar{u})\,dx + |\{ \bar{u}>0\} | \le \int_{\widetilde{\O}} j(x,\widetilde{u})\,dx + |\{\widetilde{u}>0\} |.$$
Since $\bar{u}$ is zero on $\widetilde{\O}\sm\O_\opt$ and $j(x,0) = 0$, we can write both integrals on the larger set $\widetilde{\O}$ obtaining
$$\int_{\widetilde{\O}} \big( j(x,\bar{u}) - j(x,\widetilde{u})\big) dx \le |\{ \widetilde{u}>0\} | - |\{\bar{u}>0\} |.$$
Since $\widetilde u\ge \bar u$, we can apply the hypothesis \eqref{hp.optk.g} to the left-hand side, which yields
$$-\int_{\widetilde{\O}}\bar{u}f\,dx-\left(-\int_{\widetilde{\O}}\widetilde{u}f\,dx\right)\le\frac{1}{c}\Big[\big|\{\widetilde{u}>0\}\big|-\big|\{\bar{u}>0\}\big|\Big].$$
Now since 
$$\int_{\widetilde\O}|\nabla\widetilde u|^p\,dx=\int_{\widetilde\O}f\widetilde u\,dx\qquad\text{and}\qquad \int_{\widetilde\O}|\nabla\bar u|^p\,dx=\int_{\widetilde\O}f\bar u\,dx,$$
we can re-write the above inequality as 
$$\left(\int_{\widetilde{\O}}\left(\frac1p|\nabla\bar{u}|^p-\bar{u}f\right)\,dx\right)-\left(\int_{\widetilde{\O}}\left(\frac1p|\nabla\widetilde{u}|^p - \widetilde{u}f\right)\,dx\right)\le\frac{p-1}{pc}\Big[|\{\widetilde{u}>0\}|-|\{\bar{u}>0\}|\Big].$$
Since $\widetilde{u}$ is the minimizer of the functional 
$$u\mapsto \int_{\widetilde{\O}} \left(\frac1p|\nabla{u}|^p-{u}f\right)\,dx$$
on $W^{1,p}_0(\widetilde \O)$, we get that 
$$ \left( \int_{\widetilde{\O}} \left(\frac1p|\nabla\bar{u}|^p-\bar{u}f\right)\,dx\right) - \left(\int_{\widetilde{\O}} \left( \frac1p|\nabla w|^p - wf \right) \, dx \right) \le \frac{1}{c} \Big[|\{ w>0\} | - |\{ \bar{u}>0\} | \Big].$$
In particular, for any ball $B_r\subset D$, we can take the function $w$ defined by
\[ \begin{cases}
- \Delta_p w = f & \text{in $B_r$},\\
w = \bar{u} & \text{on $D\setminus \partial B_r$}.
\end{cases}\]
Thus, the last inequality can be rewritten as 
$$\frac1p\int_{B_r} \left( |\nabla \bar{u}|^p - |\nabla w|^p \right) \, dx \le \int_{B_r} f(\bar{u}-w) \, dx + \frac{1}{c} | B_r \cap \{\bar{u}=0\}|.$$
To estimate the first integral on the right-hand side we apply H\"{o}lder inequality twice (since $f \in L^q$ is the assumption we want to use) and obtain
\[ \begin{aligned}
\int_{B_r} f(\bar{u}-w) \, dx & \le \| \bar{u}-w\|_{L^{p^\ast}} \left( \int_{B_r} f^{dp/(dp+p-d)}\right)^{1 + 1/d-1/p}
\\ & \le C \| \bar{u}-w\|_{L^{p^\ast}}\|f\|_{L^q} r^{\alpha_p},
\end{aligned} \]
where $C$ is a dimensional constant related to the volume of the unit ball in $\R^d$ and
$$\alpha_p := \frac{dp + p - d}{dp} - \frac{1}{q}.$$
Now, following \cite[Section 3]{ddap1}, we have for $p \ge 2$ the inequality
$$\int_{B_r} | \nabla (\bar{u}-w)|^p \, dx \le C \left[ \| \bar{u}-w\|_{L^{p^\ast}}\|f\|_{L^q} r^{\alpha_p} + \int_{B_r}\mathds{1}_{\{ \bar{u}=0 \}}(x) \, dx \right],$$
while for $1 < p \le 2$ we have
$$ \int_{B_r} | \nabla (\bar{u}-w)|^p \, dx \le C \left[ \| \bar{u}-w\|_{L^{2^\ast}}\|f\|_{L^q} r^{\alpha_2} + \int_{B_r} \mathds{1}_{ \{ \bar{u}=0 \}}(x)\,dx\right]^{p/2} \left[ \int_{B_r} |\nabla \bar{u}|^p \, dx \right]^{1-p/2}.$$
The conclusion now follows by a similar argument as in \cite[Lemma 3.1]{ddap1}.
\end{proof}

\section{Minimum problem on $\gamma$-compact classes}\label{scpt}

In general, without the assumptions we have seen in Sections \ref{spgd} and \ref{spld}, the existence of an optimal set $\O$ may fail if we put no geometrical restriction on the class $\A$ of admissible competitors. On the contrary, adding extra a priori geometrical constraints may lead to existence results under very general assumptions. There are numerous different classes of admissible domains in the literature, in which the existence of optimal domains is well known. In this section, we briefly recall some of the most commonly used geometric constraints. \medskip

\noindent The following classes have been considered in \cite{BB05}.
\begin{enumerate}[$-$]
\item The class $\A_{convex}$ of convex sets $\O\subset D$.\smallskip
\item The class $\A_{unif\;cone}$ of domains $\O$ satisfying a uniform exterior cone property, which means that for every point $x_0\in\partial\O$ there is a closed cone, with uniform (independent of $\O$) height and opening, and with vertex in $x_0$, lying in the complement of $\O$.\smallskip
\item The class $\A_{unif\;flat\;cone}$ of domains $\O$ satisfying a uniform flat cone condition, i.e. as above, but the cone may be $k$-flat, namely of dimension $k$, with $k>d-p$.\smallskip
\item The class $\A_{cap\;density}$ of domains $\O$ satisfying a uniform capacitary density condition, i.e. such that there exist $c,r>0$ (independent of $\O$) with the property that for every $x\in\partial\O$, we have
\[
\frac{\cp_p\left(B_t(x)\sm\O,B_{2t}(x)\right)}{\cp_p\left(B_t(x),B_{2t}(x)\right)}\ge c\qquad \text{for all $t\in(0,r)$},
\]
where $B_s(x)$ denotes the ball of radius $s$ centered at $x$, and $\cp_p$ denotes the $p$ capacity.\smallskip
\item The class $\A_{unif\;Wiener}$ of domains $\O$ satisfying a uniform Wiener condition, i.e. such that for every $x\in\partial\O$ and for every $0<r<R<1$
\[
\int_r^R\bigg[\frac{\cp_p\left(B_t(x)\sm\O,B_{2t}(x)\right)}{\cp_p\left(B_t(x),B_{2t}(x)\right)}\bigg]^{1/(p-1)}\frac{dt}{t}\ge g(r,R,x),
\]
where $g:(0,1)\times(0,1)\times D\to\R^+$ is fixed (independent of $\O$) and has the property that for every $R\in(0,1)$
\[
\lim_{r\to0}g(r,R,x)=+\infty \quad \text{locally uniformly on $x$}.
\]
\end{enumerate}

The following inclusions hold:
\[
\A_{convex}\subset\A_{unif\;cone}\subset\A_{unif\;flat\;cone}\subset\A_{cap\;density}\subset\A_{unif\;Wiener}\;.
\]
In addition, in the classes above the $\gamma$-convergence for a sequence $(\O_n)$ of domains is equivalent to the Hausdorff convergence of the sets $K_n=\overline D\sm\O_n$, and is also called Hausdorff complementary convergence, it is denoted by $H^c$ and is associated to the distance
\[
d_{H^c}(\O_1,\O_2)=d_H(\O_1^c,\O_2^c)
\]
where $d_H$ is the usual Hausdorff distance
\[
d_H(K_1,K_2)=\sup_{x\in K_1}\left[\inf_{y\in K_2}|x-y|\right]\vee\sup_{x\in K_2}\left[\inf_{y\in K_1}|x-y|\right].
\]
In particular, the properties below, which are well-known (see for instance \cite{HP18}) for the Hausdorff convergence on the class
\[
\A=\left\{\O \subset D \: : \: \text{$\O$ open} \right\}
\]
also hold for the classes above endowed with the $\gamma$-convergence:
\begin{enumerate}[label=(\roman*)]
\item $(\A,d_{H^c})$ is a compact metric space;
\item if $\O_n\to\O$ in the $H^c$ convergence, then for every compact set $K\subset\O$, there exists $n_K\in\N$ such that $K\subset\O_n$ for every $n\ge n_K$;
\item the Lebesgue measure is lower semicontinuous for the $H^c$-convergence;
\item the map that associates to a set $\O$ the number of connected components of the set $\overline D\sm\O$ is lower semicontinuous with respect to the $H^c$-convergence.
\end{enumerate}

Another interesting class, which is only of topological type and is not contained in any of the previous ones, was considered by \v{S}ver\'ak in \cite{SV93} and consists of all open subsets $\O$ of $D$ for which the number of connected components of $\overline D\sm\O$ is uniformly bounded. When for $p=d=2$ it is proved that in the class
$${\mathcal O}_k=\left\{\O\subset D\::\:\text{$\O$ open},\ \#\Big(\overline D\sm\O\Big)\le k \right\},$$
where $\#$ denotes the number of connected components, the Hausdorff complementary convergence implies the $\gamma$-convergence.

For higher dimensions the result in the form above is false in general, and capacity properties play an important role. When $p>d$ a point has a strictly positive $p$-capacity, and Theorem \ref{t:pmagd} gives the existence of optimal shapes under very mild assumptions. When $p\le d$ the generalization of the \v{S}ver\'ak result to higher dimensions was obtained in \cite{BT98} where the theorem below is proved.
 
\begin{theo}
Let $d-1<p\le d$. If a sequence $(\O_n)$ converges in the Hausdorff complementary topology to $\O$ and $\O_n$ belong to the class ${\mathcal O}_k$ above, then $\O_n$ $\gamma_p$-converges to $\O$ and $\O\in{\mathcal O}_k$.
\end{theo}

As a consequence of this theorem, a large class of shape optimization problems admit solutions in the classes ${\mathcal O}_k$ above. We summarize this in the following proposition.

\begin{coro}
Let $d-1<p\le d$ and assume that the cost integrand $j(x,s)$ satisfies the assumption:
\[
-a(x)-c|s|^q\le j(x,s)\quad\text{with }a\in L^1(D),\ c>0,\ q<\frac{dp}{d-p}\quad\text{(any $q<+\infty$ if }p=d).
\]
Then for every integer $k$ the shape optimization problem
\[
\min\left\{\int_Dj(x,u_\O)\,dx\: :\: \O\in{\mathcal O}_k\right\}
\]
admits a solution, where $u_\O$ denotes the solution of \eqref{e:stateq}.
\end{coro}

\section{Finite perimeter of optimal sets}\label{scfinper}

In this section we prove that under very mild assumptions the optimal open set $\O_\opt$ has finite perimeter. Notice that the model case
$$j(x,s)=-g(x)s+\mathds{1}_{\{s>0\}}\qquad\text{with }g(x)\ge0$$
is well within the assumptions ({\romannumeral 1}) and ({\romannumeral 2}) of Theorem \ref{t:pfin}.

\begin{proof}[Proof of Theorem \ref{t:pfin}]
Let $\O_\opt := \{ \bar{u} > 0\}$, where $\bar{u}$ is a solution of the auxiliary minimization problem
$$\min\Big\{\int_D \left( j_0(x,u) + \mathds{1}_{\{u>0\}}(x) \right) \,dx\ :\ \Delta_p u+f\ge0,\ u\in W_0^{1,p}(D)\Big\}.$$
Then, for any given $\varphi \in W_0^{1,p}(D)$ nonnegative ($\varphi \ge 0$) test function, we have
$$ \langle \Delta_p \bar{u} + f, \varphi \rangle \ge 0,$$
which, integrating by parts, leads to
$$ \left\langle f, \varphi \right\rangle \ge \left\langle |\nabla \bar{u}|^{p-2} \nabla \bar{u}, \nabla \varphi \right\rangle.$$
Let $\eps > 0$ be small enough for \eqref{eq.pf1} or \eqref{eq.pf2} to hold and set $\O_\eps := \{0 < \bar{u} < \eps\}$. If we take as a test function $\varphi := \bar{u} \wedge \eps$, then 
$$ \left\langle \nabla \bar{u}, \nabla \varphi \right\rangle = \int_{D}|\nabla \bar{u}|^{p-2} \nabla \bar{u} \cdot \nabla (\bar{u} \wedge \eps) \, dx = \int_{{\O}_\eps} |\nabla \bar{u}|^p \, dx$$
and
$$ \left\langle f, \varphi \right\rangle = \int_{D} f (\bar{u} \wedge \eps) \, dx \le \int_{{\O}_\eps} f\bar{u} \, dx + \int_{D \sm {\O}_\eps} f \eps \, dx \le C \eps.$$
Putting these two together yields
\be \label{eq.pf3}
\int_{{\O}_\eps} |\nabla \bar{u}|^p \, dx \le C \eps,
\ee
where we use the letter $C$ to denote a positive constant which may vary from line to line. Now we use the minimality of $\bar{u}$ and pick $\bar{v}:= (\bar{u}-\eps)^+$ as a competitor. It is easy to see that $\bar{v}$ satisfies the condition $\Delta_p\bar{v}+f\ge0$, so that
$$ \int_D\left( j_0(x,\bar{u})+\mathds{1}_{\{\bar{u}>0\}}(x) \right)\,dx\le\int_D\left( j_0(x,\bar{v})+\mathds{1}_{\{\bar{v}>0\}}(x) \right)\,dx. $$
Therefore we get
$$ \int_{{\O}_\eps} \left( j_0(x,\bar{u}) + 1 \right) \, dx + \int_{{\O}\sm {\O}_\eps} j_0(x,\bar{u}) \, dx \le \int_{{\O}\sm {\O}_\eps} j_0(x, \bar{u}-\eps) \, dx. $$
To use the assumptions \eqref{eq.pf1} or \eqref{eq.pf2} we move everything on the right-hand side and then add the integral of $j_0(x,0)$ on ${\O}_\eps$ on both sides obtaining
\[ \begin{aligned}
\int_{\O_\eps}1\, dx&=\int_{\O_\eps} \left( 1 + j_0(x,0) \right) \, dx
\\ & \le \int_{{\O}\sm {\O}_\eps} \left[j_0(x, \bar{u}-\eps)-j_0(x,\bar{u}) \right] \, dx + \int_{{\O}_\eps} \left[ j_0(x,0)-j_0(x,\bar{u}) \right] \, dx,
\end{aligned} \]
where the first equality follows from assumption \eqref{eq.jpos}. 

Suppose now that we are in the case ({\romannumeral 1}) of Theorem \ref{t:pfin}. Then, we can use the above estimate in the following way:
$$ \int_{{\O}_\eps} |j_0(x,0) - j_0(x,\bar{u})| \, dx \le \int_{{\O}_\eps} a(x) \bar{u} \, dx \le C\eps $$
since $\bar{u} < \eps$ on $\O_\eps$ by definition and $a \in L^1(D)$; similarly, we have
$$ \int_{{\O}\sm {\O}_\eps} \left| j_0(x, \bar{u}-\eps)-j_0(x,\bar{u}) \right| \, dx \le \int_{{\O}\sm {\O}_\eps} \eps \left[ a(x) + c|\bar{u}|^{p^\ast} \right] \, dx \le C\eps, $$
as a consequence of the fact that for $u \in W_0^{1,p}(D)$ we can apply Sobolev embedding theorem. On the other hand, on the left-hand side we have
$$ \int_{{\O}_\eps}1\, dx = |\O_\eps|, $$
so putting everything together yields
\be \label{eq.pf4}
| \bar{\O}_\eps | \le C\eps. 
\ee
If we are in the case ({\romannumeral 2}), then we simply notice that $\bar{u}$ is bounded (it follows from Theorem \ref{t:pmagd} for $p > d$ and Lemma \ref{lemma.bdd1} for $p\le d$); therefore, the estimates are
\[ \begin{aligned}
& \int_{{\O}_\eps} |j_0(x,0) - j_0(x,\bar{u})| \, dx \le \int_{\O_\eps}a(x,\bar{u}) \bar{u} \, dx \le \eps C\|\bar{u}\|_\infty,
\\ &\int_{{\O}\sm {\O}_\eps} \left| j_0(x, \bar{u}-\eps)-j_0(x,\bar{u}) \right| \, dx \le \int_{{\O}\sm {\O}_\eps} \eps \left[ a(x,\bar{u}) \right] \, dx \le \eps C\|\bar{u}\|_\infty,
\end{aligned}\]
which means that only the constants in front of $\eps$ are different from case ({\romannumeral 1}). 

If we now put \eqref{eq.pf3} and \eqref{eq.pf4} together and apply H\"older's inequality we get
$$ \int_{\O_\eps}|\nabla\bar{u}|\,dx \le \left[\int_{\O_\eps}|\nabla\bar{u}|^p\,dx \right]^{1/p}|\O_\eps|^{1/p'}\le (C\eps)^{1/p}(C\eps)^{1/p'} = C\eps, $$
so we can conclude that ${\O}$ has finite perimeter using the argument from \cite{bulbk}. More precisely, the coarea formula gives
$$ \int_0^\eps \HH^{d-1}(\partial^\ast\{\bar{u}>\eps\})\,dt\le C\eps $$
and hence there exists $\delta_n \xrightarrow{n\to \infty} 0$ such that
$$\HH^{d-1}(\partial^\ast\{\bar{u}>\delta_n\})\le C\qquad\text{for all }n\in\N,$$
from which it follows that
$$\per({\O})=\HH^{d-1}(\partial^\ast\{\bar{u}>0\})\le C,$$
and this concludes the proof.
\end{proof}

\bigskip\noindent{\bf Acknowledgments.} The work of GB is part of the project 2017TEXA3H {\it``Gradient flows, Optimal Transport and Metric Measure Structures''} funded by the Italian Ministry of Research and University. The first author is member of the Gruppo Nazionale per l'Analisi Matematica, la Probabilit\`a e le loro Applicazioni (GNAMPA) of the Istituto Nazionale di Alta Matematica (INdAM). BV has been supported by the European Research Council (ERC), under the European Union’s Horizon 2020 research and innovation programme, through the project ERC VAREG - \it Variational approach to the regularity of the free boundaries \rm (grant agreement No. 853404).

\bigskip
{\small\noindent
Giuseppe Buttazzo:
Dipartimento di Matematica, Universit\`a di Pisa\\
Largo B. Pontecorvo 5, 56127 Pisa - ITALY\\
{\tt giuseppe.buttazzo@unipi.it}\\
{\tt http://www.dm.unipi.it/pages/buttazzo/}

\bigskip\noindent
Francesco Paolo Maiale:
Scuola Normale Superiore\\
Piazza dei Cavalieri 7, 56126 Pisa - ITALY\\
{\tt francesco.maiale@sns.it}\\
{\tt https://sites.google.com/site/francescopaolomaiale/}

\bigskip\noindent
Bozhidar Velichkov:
Dipartimento di Matematica, Universit\`a di Pisa\\
Largo B. Pontecorvo 5, 56127 Pisa - ITALY\\
{\tt bozhidar.velichkov@unipi.it}\\
{\tt http://www.velichkov.it/index.html}

\end{document}